\documentclass{amsart}

\usepackage{amsmath,amssymb,amsfonts,enumerate,amsthm,graphicx,color,hyperref}

\newtheorem{thm}{Theorem}[section]
\newtheorem{cor}[thm]{Corollary}
\newtheorem{lem}[thm]{Lemma}
\newtheorem{prop}[thm]{Proposition}
\newtheorem{defn}[thm]{Definition}

\newtheorem{exam}[thm]{Example}
\newtheorem{rem}[thm]{Remark}

\numberwithin{equation}{section}

\begin{document}

\bibliographystyle{amsplain}

\author{Amir Mousivand}

\address{Amir Mousivand\\Department of Mathematics, Firoozkooh Branch, Islamic
Azad University (IAU), Firoozkooh, Iran.}\email{amirmousivand@gmail.com}

%\thanks{}



\title{Circulant $S_2$ graphs}

\begin{abstract}
Recently, Earl, Vander Meulen, and Van Tuyl characterized some families of Cohen-Macaulay or Buchsbaum circulant graphs discovered by Boros-Gurvich-Milani$\check{\text{c}}$, Brown-Hoshino, and Moussi. In this paper, we will characterize those families of circulant graphs which satisfy Serre's condition $S_2$. More precisely, we show that for some families of circulant graphs, $S_2$ property is equivalent to well-coveredness or Buchsbaumness, and for some other families it is equivalent to Cohen-Macaulayness.  We also give examples of infinite families of circulant graphs which are Buchsbaum but not $S_2$, and vice versa.
\end{abstract}

\maketitle

\section{Introduction}

Let $G$ be a finite simple undirected graph with the vertex set $V(G)$ and the edge set $E(G)$. A subset $C$ of $V(G)$ is called a {\it vertex cover} of $G$
if $C  \cap e \ne \emptyset $ for  any  $e \in E(G)$. A vertex cover $C$ of $G$ is called \textit{minimal} if there is no proper subset of $C$
which is a vertex cover. A graph $G$ is said to be {\it well-covered} if all its minimal vertex covers have the same cardinality.

A simplicial complex $\Delta$ on the vertex set  $V=\{x_1 , \ldots, x_n \}$ is a collection of subsets of $V$, with the properties:
(1) $\{x_i\}\in\Delta$ for all $i$, and (2) if $F\in\Delta$, then all subsets of $F$ are also in $\Delta$ (including the empty set).
An element $F$ of $\Delta$ is called a {\it face} of $\Delta$, and maximal faces of $\Delta$ (with respect to inclusion) are called {\it facets} of $\Delta$. We denote the simplicial complex $\Delta$ with facets $F_1 , \ldots , F_t$ by $\Delta = \langle F_1 , \ldots , F_t \rangle$.
The {\it dimension} of a face $F\in\Delta$ is defined by $\dim F= | F | -1$, and The dimension of $\Delta $ is defined by $\dim \Delta =  \max \{ \dim F ~:~ F \in \Delta \}$. A simplicial complex is called {\it pure} if all its facets have the same cardinality. $\Delta$ is called {\it Cohen-Macaulay} (resp. {\it Buchsbaum}) over a field $k$ if its Stanley-Reisner ring $k[\Delta]$ is Cohen-Macaulay (resp. Buchsbaum), and called Cohen-Macaulay (resp. Buchsbaum) if it has the same property over any field $k$.
For a face $F$ of $\Delta$, the {\it link} of $F$ is the simplicial complex
$$\textnormal{link}_\Delta(F)=\{G\in\Delta ~ ~:~ ~ G\cap F=\emptyset ~ \text{and} ~ G\cup F\in \Delta\}.$$
By Reisner's criterion (see e.g. \cite [Theorem 5.3.5]{V}), A simplicial complex $\Delta$ is Cohen-Macaulay over a field $k$, if and only if $\tilde{H}_i(\textnormal{link}_\Delta(F);k)=0$ for all $F\in \Delta$ and $i<\rm{dim}\rm{link}_\Delta(F)$.

The {\it independence complex} of a graph $G$, denoted by $\text{Ind}(G)$, is the simplicial complex whose faces correspond to independent (or stable) sets of $G$, where a subset $F$ of $V(G)$ is called an independent set if any subsets of $F$ with cardinality two do not belong to $E(G)$. Since the complement of a vertex cover is an independent set, it follows that a graph $G$ is well-covered if and only if $\text{Ind}(G)$ is a pure simplicial complex. A graph $G$ is called Cohen-Macaulay (resp. Buchsbaum) if the independence complex $\text{Ind}(G)$ is Cohen-Macaulay (resp. Buchsbaum).

Given an integer $n\geq 1$ and a generating set $S\subseteq \{1,2\ldots,\lfloor\frac{n}{2}\rfloor\}$, the {\it circulant graph} $C_n(S)$ is the graph with the vertex set $V=\{0,1,\ldots,n-1\}$ whose edge set is $$E=\{\{i,j\} ~ ~ : ~ ~ |i-j|\in S ~ ~ \text{or} ~ ~ n-|i-j|\in S\}.$$
For $S=\{a_1,\ldots,a_t\}$, we abuse the notation and use $C_n(a_1,\ldots,a_t)$ to denote $C_n(S)$. Circulant graphs belong to the family of cayley graphs and may be considered as a generalization of cycles because $C_n=C_n(1)$.
In recent years, there have been a flurry of work identifying circulant graphs which are also well-covered (see e.g. \cite{BGM,BH1,BH2,H,M}).
Since a well-covered graph has the property that its independence complex is pure, and a pure complex can have some extra combinatorial (e.g. vertex decomposable and shellable) or topological (e.g. Cohen-Macaulay and Buchsbaum) structure, i.e., for a pure simplicial complex the following hierarchy  is known:
$$\mbox {vertex~decomposability}
\Longrightarrow
\mbox {shellability}
\Longrightarrow
\mbox {Cohen-Macaulayness}
\Longrightarrow
\mbox {Buchsbaum},
$$
so it is natural to ask what more structures $\text{Ind}(C_n(S))$ entertains?

Recently, Vander Meulen, Van Tuyl, and Watt \cite{VVW} characterized Cohen-Macaulay (vertex decomposable, shellable, or buchsbaum)  circulant graphs of the form $C_n(1,2,\ldots,d)$ and Cohen-Macaulay cubic circulant graphs. Earl, Vander Meulen, and Van Tuyl \cite{EVV} determined when circulant graphs of the form $C_n(d + 1, d + 2,\ldots ,\lfloor\frac{n}{2}\rfloor)$, $C_n(1,\ldots,\hat{i},\ldots,\lfloor\frac{n}{2}\rfloor)$, and one-paired circulants have these structures. Also Vander Meulen and Van Tuyl \cite{VV} investigated when the independence complex of the lexicographical product of two graphs is either vertex decomposable or shellable. They also constructed an infinite family of graphs in which the independence complex of each graph is shellable, but not vertex decomposable.

A finitely generated graded module $M$ over a Noetherian graded ring $R$ is said to satisfy the Serre's condition $S_r$ (or simply say $M$ is an $S_r$ module) if
$$\text{depth}(M_p)\geq \text{min}\{r, \text{dim}(M_p)\},$$ for all $p\in \text{Spec}(R)$. Since $M$ is Cohen-Macaulay if $\text{depth}(M_p)=\text{dim}(M_p)$ for every $p\in \text{Spec}(R)$, it follows that $M$ is Cohen-Macaulay if and only if it satisfies the Serre's condition $S_r$ for all $r\geq 1$. A simplicial complex $\Delta$ is said to satisfy Serre's condition $S_r$ over a field $k$ (or simply say $\Delta$ is an $S_r$ complex) if the Stanley-Reisner ring $k[\Delta]$ satisfies Serre's condition $S_r$. Terai \cite{T} presented the following analogue of Reisner's criterion for $S_r$ simplicial complexes.\\\\
{\bf Theorem.}
{\it A simplicial complex $\Delta$ satisfies Serre's condition $S_r$ over a field $k$ if and only if for every face $F\in\Delta$ (including the empty face), $\tilde{H}_i(\textnormal{link}_\Delta(F);k)=0$ for all $i<\rm{min} \{{\it r}-1,\rm{dimlink}_\Delta(F)\}$.}\\

There are some basic facts related to $S_r$ simplicial complex. Every simplicial complex satisfies $S_1$. On the
other hand, for $r\geq 2$, simplicial complexes satisfying $S_r$ (over a field k), are pure (\cite[Lemma 2.6]{MT}) and strongly connected (\cite [Corollary 2.4]{Ha}). Recall that a pure simplicial complex $\Delta$ is called strongly connected if, for every pair of facets $F$ and $F'$ of $\Delta$, there
exists a sequence of facets $F=F_0,F_1,\ldots,F_t =F'$ such that $|F_k\cap F_{k+1}|=|F_k|-1$ for $k=0,1,\ldots,t-1$.
We refer the reader to \cite{GPSY,HTYZ,MT,PSTY,T,TY} for more details on the properties of $S_r$ (and sequentially $S_r$) simplicial complexes.
Also as an immediate consequence of previous Theorem, we have the following corollary on $S_2$ simplicial complexes.\\\\
{\bf Corollary.}
{\it A simplicial complex $\Delta$ satisfies Serre's condition $S_2$ over a field $k$ if and only if $\textnormal{link}_\Delta(F)$ is connected for every face $F\in\Delta$ with $\rm{dim}\rm{link}_\Delta(F)\geq 1$. In particular, $S_2$ property of a simplicial complex does not depend on the characteristic of the field $k$.}\\

Although Cohen-Macaulay (or Buchsbaum) property of stanley-Reisner rings depend on the base field and hence it is a topological property, $S_2$ property does not depend on the base field and $S_2$ Stanley-Reisner rings can be characterized combinatorially.

We say a graph $G$ satisfies the Serre's condition $S_n$, or simply is an $S_n$ graph, if its independence complex $\text{Ind}(G)$ satisfies this condition. As an immediate consequence, it follows that $S_n$ graphs are well-cowered. Haghighi, Yassemi, and Zaare-Nahandi \cite{HYZ} showed that $S_2$ property for bipartite graphs and chordal graphs is equivalent to Cohen-Macaulayness. Recall that a graph $G$ is bipartite if there exists a partition $V(G)=V \bigcup {V'}$ with $V \bigcap {V'} = \emptyset$ such that each edge of $G$ is of the form $\{i,j\}$ with $i \in V$ and $j \in {V'}$, and $G$ is called chordal if every cycle of length at least four has a chord, where a chord is an edge joining two nonadjacent vertices of the cycle.

In this paper we characterize some families of circulant $S_2$ graphs. In section 2, we consider $S_2$ property of powers of cycles. More precisely, we show that if $n\geq 2d\geq 2$, then $C_n(1,2,\ldots,d)$ is $S_2$ if and only if $n\leq 3d+2$ and $n\neq 2d+2$, or $n=4d+3$. Comparing this with \cite[Theorem 3.7]{VVW} implies that $C_n(1,2,\ldots,d)$ is Buchsbaum but not $S_2$ if and only if $n=2d+2$. In section 3, we consider circulants of the form $C_n(d + 1, d + 2,\ldots ,\lfloor\frac{n}{2}\rfloor)$. We show that such a graph satisfies serre's condition $S_2$ if and only if it is well-covered, i.e., $n>3d$ or $n= 2d + 2$. This is equivalent to say that $C_n(d + 1, d + 2,\ldots ,\lfloor\frac{n}{2}\rfloor)$ is Buchsbaum. In this case, we will show that the only non-shellable (connected) link in $\Delta=\text{Ind}(C_n(d + 1, d + 2,\ldots ,\lfloor\frac{n}{2}\rfloor))$ is the link of $\emptyset$ which is $\Delta$ (except for $d=1$), and the link of all non-empty faces of $\Delta$ are shellable. In section 4, we investigate circulants of the form $C_n(1,\ldots,\hat{i},\ldots,\lfloor\frac{n}{2}\rfloor)$. We show that cirulant $S_2$ graphs in this family are all Cohen-Macaulay. Using \cite[Theorem 4.2]{EVV} it is equivalent to say that $\text{gcd}(i,n)=1$. Since circulants in this family are all Buchsbaum (see \cite[Theorem 4.2]{EVV}), it follows that circulant graphs of the form $C_n(1,\ldots,\hat{i},\ldots,\lfloor\frac{n}{2}\rfloor)$ for which $\text{gcd}(i,n)> 1$, is an infinite family of Buchsbaum graphs that are not $S_2$. Section 5 is dedicated to one-paired circulant graph, where we show that one-paired circulant graph $C(n;a,b)$ is $S_2$ if and only if $n=ab$, i.e., it is Cohen-Macaulay. Combining this together with \cite[Corollary 5.10]{EVV} implies that, if $m>1$, then $C(mb;1,b)$ is Buchsbaum but not $S_2$. Finally, in section 6 we identify which circulant cubic graphs are $S_2$. More precisely, we first show that the only non-$S_2$ connected circulant cubic graph is $C_6(1,3)$. Combining this together with a result of Davis and Domke \cite{DD} enables us to prove that the circulant cubic graph $C_{2n}(a,n)$ is $S_2$ if and only if $\frac{2n}{t}=3,4,5,8$, where $t={\rm gcd}(a,2n)$. As the final result of this paper, we present infinite families of circulant graphs which are $S_2$ but not Buchsbaum, namely, circulants of the form $C_{8t}(t,4t)$, $C_{10t}(2t,5t)$, and $C_{10t}(4t,5t)$, where $t>1$.

\section{Circulants of the form $C_n(1,2,\ldots,d)$}

In this section we identify which circulants of the form $C_n(1,2,\ldots,d)$ satisfy serre's condition $S_2$. We need the following result of Brown and Hoshino.

\begin{thm} \label{THM2}
{\rm (\cite[Theorem 4.1]{BH2})}  Let $n$ and $d$ be integers with $n\geq 2d\geq 2$. Then $C_n(1,2,\ldots,d)$ is well-covered if and only if $n\leq 3d+2$ or $n=4d+3$.
\end{thm}

Using the above result we get the next characterization of circulant $S_2$ graphs of the form $C_n(1,2,\ldots,d)$.

\begin{thm} \label{THM3} Let $n$ and $d$ be integers with $n\geq 2d\geq 2$. Then $G=C_n(1,2,\ldots,d)$ is $S_2$ if and only if $n\leq 3d+2$ and $n\neq 2d+2$, or $n=4d+3$.
\end{thm}

\begin{proof} If $G$ is $S_2$, then $G$ is well-covered and hence by Theorem \ref{THM2} we get $n\leq 3d+2$ or $n=4d+3$. In the case where $n=2d+2$, $\text{Ind}(G)$ comprises of $d+1$ disjoint $1$-faces (edges) and hence $G$ is not $S_2$. Now we prove the converse. The idea is inspired by the proof of \cite[Theorems 3.4 and 3.5]{VVW}. By \cite[Theorem 3.1]{BH2} one has $\text{dimInd}(G)=\lfloor \frac{n}{d+1}\rfloor-1$. If $n=2d$ or $n=2d+1$, then $\text{dimInd}(G)=0$ and $G$ is Cohen-Macaulay. If $2d+3\leq n\leq 3d+2$, then $\text{dimInd}(G)=1$. In this case for $0\leq i<j\leq n-1$, there exists the path $i,i+d+2,i+1,i+d+3,i+2,\ldots,j$ of $1$-faces (facets or edges) of $\text{Ind}(G)$. In particular, $\text{Ind}(G)$ is connected, and again it is Cohen-Macaulay. Now assume $n=4d+3$. One has $\text{dimInd}(G)=2$. First note that
$$\textit{\bf{0}},d+1,2d+2,3d+3,\textit{\bf{1}},d+2,2d+3,3d+4,\textit{\bf{2}},\ldots,\textbf{\textit{d}-1},2d,3d+1,4d+2,\textbf{\textit{d}},2d+1,3d+2,\textit{\bf{0}}$$
is a path (cycle) of $1$-faces of $\text{Ind}(G)$, i.e., $\text{Ind}(G)$ is connected. To complete the proof, by symmetry, it suffices to show that $\rm{link}_{\rm{Ind}(G)}(0)$ is connected. One can directly check that \\

$\{0,d+1,2d+2\}, \hspace{3mm} \{0,d+1,2d+3\}, \hspace{3mm} \ldots, \hspace{3mm} \{0,d+1,3d+1\}, \hspace{3mm} \{0,d+1,3d+2\},$

$\{0,d+2,2d+3\}, \hspace{2mm} \{0,d+2,2d+4\}, \hspace{3mm} \ldots, \hspace{2mm} \{0,d+2,3d+2\},$

$\hspace{9mm} \vdots \hspace{28mm} \vdots$

$\{0,2d,3d+1\}, \hspace{3mm} \{0,2d,3d+2\},$

$\{0,2d+1,3d+2\}$\\
is a complete list of facets of $\text{Ind}(G)$. From this description quickly follows that $\rm{link}_{\rm{Ind}(G)}(0)$ is connected, and $G$ is $S_2$.
\end{proof}

As an immediate consequence, we get the following corollary which is the content of \cite[Proposition 1.6]{HYZ}.

\begin{cor} The cyclic graph $C_n$ of length $n\geq 3$ is $S_2$ if and only if $n = 3,5,$ or $7$. In particular, $C_7$ is the only cyclic graph which is $S_2$ but not Cohen-Macaulay.
\end{cor}

\begin{exam} {\rm Vander Meulen {\it et al.} in \cite[Theorems 3.4 and 3.5]{VVW} showed that for $n\geq 2d\geq 2$ one has
\begin{itemize}
\item[(i)] $C_n(1,2,\ldots,d)$ is vertex decomposable/shellable/Cohen-Macaulay if and only if $n\leq 3d+2$ and $n\neq 2d+2$.
\item[(ii)] $C_n(1,2,\ldots,d)$ is Buchsbaum but not Cohen-Macaulay if and only if $n=2d+2$ or $n=4d+3$.
\end{itemize}
Comparing these with Theorem \ref{THM3} yields that $C_n(1,2,\ldots,d)$ is Buchsbaum but not $S_2$ if and only if $n=2d+2$.
}
\end{exam}

\section{Circulants of the form $C_n(d + 1, d + 2,\ldots ,\lfloor\frac{n}{2}\rfloor)$}

In this section we investigate $S_2$ circulant graphs of the form $C_n(d + 1, d + 2,\ldots ,\lfloor\frac{n}{2}\rfloor)$. To do this, we need the following result of Brown and Hoshino on well-coveredness of these circulant graphs.

\begin{thm} \label{THM1}
{\rm (\cite[Theorem 4.2]{BH2})} Let $n $ and $d$ be integers with $n\geq 2d+2$ and $d\geq1$. Then $C_n(d + 1, d + 2,\ldots ,\lfloor\frac{n}{2}\rfloor)$ is well-covered if and only if $n>3d$ or $n= 2d + 2$.
\end{thm}

\begin{defn} {\rm A simplicial complex $\Delta$ is called {\it shellable} if there is a linear order
$F_1 , \ldots, F_s$ of all the facets of $\Delta $ such that for all
$1\leq i<j\leq s$, there exists some $v\in F_j\setminus F_i$ and
some $l\in \{1 , \ldots, j-1\}$ with $F_j\setminus F_l=\{v\}$. $F_1 , \ldots, F_s$ is called a shelling order of $\Delta$.}
\end{defn}

We will make use the next two Lemmas to prove Theorem \ref{MAIN1} which is the main result of this section.

\begin{lem} \label{LEM1}
Let $\Delta$ be a $d$-dimensional pure simplicial complex on the vertex set $V=\{x_1,x_2,\ldots,x_n\}$ whose facets are given by
$$F_i=\{x_i,x_{i+1},\ldots,x_{i+d}\}\hspace{3mm}:\hspace{3mm}i=1,2,\ldots,n-d.$$
Then $\Delta$ is pure shellable.
\end{lem}

\begin{proof} It is easy to see that $F_1,F_2,\dots,F_{n-d}$ is the desired shelling order on the facets of $\Delta$.
\end{proof}

\begin{lem} \label{LEM2} Let $\Delta$ be a shellable simplicial complex and $F_1,F_2,\dots,F_s$ a shelling order of $\Delta$. Also let $F\in\Delta$ be such that $F\subseteq F_i$ for all $i=1,2,\ldots,s$. Then $\rm{link}_\Delta(F)$ is shellable.
\end{lem}

\begin{proof} It is clear that $\text{link}_\Delta(F)=\langle F_1\setminus F,F_2\setminus F,\dots,F_s\setminus F\rangle$, and that this order is the desired shelling order of the facets of $\text{link}_\Delta(F)$.
\end{proof}

We also need the following result on circulants of the form $C_n(d + 1, d + 2,\ldots ,\lfloor\frac{n}{2}\rfloor)$.

\begin{thm} \label{THM4}
{\rm (\cite[Theorem 3.3]{EVV})} Let $n $ and $d$ be integers with $n\geq 2d+2$ and $d\geq1$. The following are equivalent:
\begin{itemize}
\item[(i)] $C_n(d + 1, d + 2,\ldots ,\lfloor\frac{n}{2}\rfloor)$ is Buchsbaum;
\item[(ii)] $C_n(d + 1, d + 2,\ldots ,\lfloor\frac{n}{2}\rfloor)$ is well-covered;
\item[(iii)] $n>3d$ or $n= 2d + 2$.
\end{itemize}
Furthermore, $C_n(d + 1, d + 2,\ldots ,\lfloor\frac{n}{2}\rfloor)$ is vertex decomposable/shellable/Cohen-Macaulay if and only if $n=2d+2$ and $d\geq 1$ or $d=1$ and $n\geq 3$.
\end{thm}
Now we are ready to state and prove the main result of this section.

\begin{thm} \label{MAIN1}
Let $n $ and $d$ be integers with $n\geq 2d+2$ and $d\geq1$. Then $C_n(d + 1, d + 2,\ldots ,\lfloor\frac{n}{2}\rfloor)$ is $S_2$ if and only if $n>3d$ or $n= 2d + 2$.
\end{thm}

\begin{proof} Let $G=C_n(d + 1, d + 2,\ldots ,\lfloor\frac{n}{2}\rfloor)$ with $n\geq 2d+2$ and $d\geq1$. If $G$ is $S_2$, then $G$ is well-covered and the result follows from Theorem \ref{THM1}. Now we prove the converse. If $n= 2d + 2$, then $G=C_{2d+2}(d + 1)$ is the union of $(d+1)$ disjoint edges. It follows that $G$ is Cohen-Macaulay, and hence it is $S_2$. So assume $n>3d$ and $d\geq 1$, and let $\Delta=\text{Ind}(G)$. By the proof of \cite[Theorem 3.3]{EVV} one has $\text{Ind}(G)=\langle F_0,\ldots,F_{n-1}\rangle$, where $F_i=\{i,i+1,\ldots,i+d\}$ for all $i=0,\ldots,n-1$, and the indices computed modulo $n$. If $\text{dim}\text{link}_\Delta(F)=d$, then $\text{link}_\Delta(F)=\Delta$, which is a connected simplicial complex that is not shellable except for $d=1$ (see Theorem \ref{THM4}) and there is nothing to prove. So it is enough to show that $\text{link}_\Delta(F)$ is connected for each face $F\in\Delta$ with $0<\text{dim}\text{link}_\Delta(F)<d$. We show that $\text{link}_\Delta(F)$ is indeed shellable.\\
If $d=1$, then the only case where $\text{dim}\text{link}_\Delta(F)>0$ is for $F=\emptyset$ which in this case $\text{link}_\Delta(F)=\Delta$ is shellable. So suppose $d>1$. Assume that $\text{dim}\text{link}_\Delta(F)=d-t$ where $0<t<d$. It follows that $|F|=t$. Without loos of generality we may assume $F\subseteq F_0$. Suppose $F=\{j_1,j_2,\ldots,j_t\}$ with $0\leq j_1<j_2<\ldots<j_t\leq d$. We claim that
$$F_{n-d+j_t},F_{n-d+j_t+1},\ldots,F_{j_1-1},F_{j_1}$$
is a complete list of the facets of $\text{Ind}(G)$ that contain $F$, where the indices computed modulo $n$. To see this, It suffices to notice that one has

$F_{n-d+j_t}=\{n-d+j_t,n-d+j_t+1,\ldots,j_t-1,j_t\},$

$F_{n-d+j_t+1}=\{n-d+j_t+1,\ldots,j_t,j_t+1\},$

$~~~\vdots$

$F_{j_1-1}=\{j_1-1,j_1,\ldots,j_1+d-1\},$

$F_{j_1}=\{j_1,j_1+1,\ldots,j_1+d\}.$\\
It follows from Lemma \ref{LEM1} that $\Delta'=\langle F_{n-d+j_t},F_{n-d+j_t+1},\ldots,F_{j_1-1},F_{j_1}\rangle$ is a pure shellable simplicial complex on the vertex set $V'=\{n-d+j_t,n-d+j_t+1,\ldots,j_1+d\}$ (Note that $|V'|=2d-(j_t-j_1)+1\leq 2d+1<n$, and that $n-d+j_t,n-d+j_t+1,\ldots,j_1+d$ are all distinct since $j_1+d\leq 2d$ and $n-d+j_t>2d+j_t\geq 2d$). Set
$$F'_0=F_{n-d+j_t}\setminus F \hspace{3mm},\hspace{3mm}  F'_1=F_{n-d+j_t+1}\setminus F \hspace{3mm},\hspace{3mm}\ldots \hspace{3mm},\hspace{3mm}F'_{d-(j_t-j_1)+1}=F_{j_1}\setminus F.$$
One can easily check that
$$\text{link}_\Delta(F)=\langle F'_0,F'_1,\ldots,F'_{d-(j_t-j_1)+1}\rangle.$$
Now it follows from Lemma \ref{LEM2} that $\text{link}_\Delta(F)$ is shellable, as desired.
\end{proof}

\begin{rem} {\rm Let $G=C_n(d + 1, d + 2,\ldots ,\lfloor\frac{n}{2}\rfloor)$, where $n>3d$ and $d>1$. Proof of Theorem \ref{MAIN1} shows that if $\Delta =\rm{Ind}(G)$, then the only non-shellable (connected) link in $\Delta$ is the link of $\emptyset$ which is $\Delta$, and the link of all non-empty faces of $\Delta$ are shellable.}
\end{rem}

\begin{cor} Let $n $ and $d$ be integers with $n\geq 2d+2$ and $d\geq1$. For $G=C_n(d + 1, d + 2,\ldots ,\lfloor\frac{n}{2}\rfloor)$ and $\Delta=\rm{Ind}(G)$ the following are equivalent:
\begin{itemize}
\item[(i)] $G$ is $S_2$;
\item[(ii)] $G$ is Buchsbaum;
\item[(iii)] $G$ is well-covered;
\item[(iv)] $n>3d$ or $n= 2d + 2$;
\item[(v)] $\Delta$ is strongly connected and $\rm{link}_\Delta(F)$ is (pure) shellable for all $F\in\Delta$ with $\rm{dim}\rm{link}_\Delta(F)<d$;
\item[(vi)]  $\rm{link}_\Delta(F)$ is strongly connected for all $F\in\Delta$ with $0<\rm{dim}\rm{link}_\Delta(F)\leq d$.
\end{itemize}
\end{cor}

\begin{proof} The equivalent of (ii), (iii), and (iv) is Theorem \ref{THM4}. Also (i) and (iv) are equivalent by Theorem \ref{MAIN1}. On the other hand, (v) and (vi) follow from the description of facets of $\Delta$ and $\rm{link}_\Delta(F)$ in the proof of Theorem \ref{MAIN1}. Finally, if (v) or (vi) hold, then $\Delta$ is pure and $\rm{link}_\Delta(F)$ is connected for all $F\in\Delta$ with $\rm{dim}\rm{link}_\Delta(F)>0$, i.e., $\Delta$ is $S_2$.
\end{proof}

\section{Circulants of the form $C_n(1,\ldots,\hat{i},\ldots,\lfloor\frac{n}{2}\rfloor)$}

Moussi \cite[Theorem 6.4]{M} proved that Circulants of the form $C_n(1,\ldots,\hat{i},\ldots,\lfloor\frac{n}{2}\rfloor)$ are well-covered. Earl, Vander Meulen, and Van Tuyl \cite[Section 4]{EVV} showed that these circulants are always Buchsbaum, and they are Cohen-Macaulay if and only if $\text{gcd}(i,n)=1$. In this section we show that this condition is equivalent to $S_2$ property.

\begin{thm} Let $G=C_n(1,\ldots,\hat{i},\ldots,\lfloor\frac{n}{2}\rfloor)$. The following are equivalent:
\begin{itemize}
\item[(i)] $G$ is $S_2$;
\item[(ii)] $G$ is Cohen-Macaulay;
\item[(iii)] ${\rm gcd}(i,n)=1$.
\end{itemize}
\end{thm}

\begin{proof} The equivalent of (ii) and (iii) is \cite[Theorem 4.2]{EVV}. Also (ii) $\Longrightarrow (i)$ always hold. Now we show that (i) $\Longrightarrow (ii)$. By \cite[Theorem 6.4]{M} one has $\text{dimInd}(G)=1$ except if $i=\frac{n}{3}$, in which case, $\text{dimInd}(G)=2$. If $i\neq\frac{n}{3}$, then by the proof of \cite[Theorem 4.2]{EVV} $\text{Ind}(G)$ is connected and hence it is Cohen-Macaulay. Now assume $i=\frac{n}{3}$. Again by the proof of \cite[Theorem 4.2]{EVV} we have $$\text{Ind}(G)=\langle\{0,i,2i\},\{1,i+1,2i+1\},\ldots,\{i-1,2i-1,3i-1\}\rangle.$$
Since $G$ is $S_2$, $\text{Ind}(G)$ is connected. This yields that $i=1, n=3$, and $\text{Ind}(G)=\langle\{0,i,2i\}\rangle$, i.e., $G$ is Cohen-Macaulay.
\end{proof}

\begin{rem} {\rm While $S_2$ property for the family $C_n(d + 1, d + 2,\ldots ,\lfloor\frac{n}{2}\rfloor)$ is equivalent to Buchsbaumness (and hence there exist circulants in this family that are $S_2$ but not Cohen-Macaulay (see Theorem \ref{THM4})), but $S_2$ graphs of the form $C_n(1,\ldots,\hat{i},\ldots,\lfloor\frac{n}{2}\rfloor)$ are those which are Cohen-Macaulay.}
\end{rem}

\begin{exam} {\rm It follows form \cite[Theorem 4.2]{EVV} that circulant graphs of the form $C_n(1,\ldots,\hat{i},\ldots,\lfloor\frac{n}{2}\rfloor)$ for which ${\rm gcd}(i,n)\neq 1$, is an infinite family of Buchsbaum graphs that are not $S_2$.}
\end{exam}

\section{One-Paired Circulants}

In this section we consider $S_2$ property of one-paired circulant graphs introduced by Boros, Gurvich, and Milani$\check{\text{c}}$ \cite{BGM} as a subfamily of CIS circulant graphs. one-paired circulant graphs define as follows.

\begin{defn} {\rm Let $(a,b)$ a pair of positive integers such that $ab|n$ and let $S=\{d\in\{1,\ldots,\lfloor\frac{n}{2}\rfloor\}~ ~ : ~ ~ a|d ~ \text{and} ~ ab\nmid d\}$ . The circulant graph $C_n(S)$ is called {\it one-paired} and will be denoted by $C(n;a,b)$}
\end{defn}

Earl, Vander Meulen, and Van Tuyl characterized the structure of one-paired circulant graphs, and using that, they identified when a one-paired circulant graph is Cohen-Macaulay or Buchsbaum. More precisely, they showed the followings.

\begin{thm} \label{THM}
{\rm (\cite[Theorem 5.4]{EVV})}
Let $G=C(n;a,b)$ be a one-paired circulant. Then
$$G=\bigcup_{i=1}^a\left(\bigvee_{j=1}^b\overline{K_{\frac{n}{ab}}}\right).$$
\end{thm}

\begin{cor} \label{COR}
{\rm (\cite[Corollary 5.10]{EVV})}
Let $G$ be the one-paired circulant graph $G=C(n;a,b)$. Then
\begin{itemize}
\item[(i)] $G$ is vertex decomposable/shellable/Cohen-Macaulay if and only if $n=ab$.
\item[(ii)] $G$ is Buchsbaum but not Cohen-Macaulay if and only if $a=1$ and $ab<n$.
\item[(iii)] $\text{Ind}(G)$ is pure but not Buchsbaum if and only if $1<a$ and $ab<n$.
\end{itemize}
\end{cor}

Using the above results we show that one-paired circulant $S_2$ graphs are Cohen-Macaulay.

\begin{thm} Let $G$ be the one-paired circulant graph $G=C(n;a,b)$. Then $G$ is $S_2$ if and only if $n=ab$, i.e., $G$ is Cohen-Macaulay.
\end{thm}

\begin{proof} By Theorem \ref{THM} one has
$$G=\bigcup_{i=1}^a\left(\bigvee_{j=1}^b\overline{K_{\frac{n}{ab}}}\right).$$
It concludes that
$$\text{Ind}(G)=\text{Ind}(G_1)\ast\text{Ind}(G_2)\ast\cdots\ast\text{Ind}(G_a),$$
where $\ast$ denotes the join of complexes, and $G_i=\bigvee_{j=1}^b\overline{K_{\frac{n}{ab}}}$ for all $i=1,\ldots,a$. Therefore each $\text{Ind}(G_i)$ is a pure complex consists of disjoint union of $b$ simplices with cardinality $k=\frac{n}{ab}$. Thus we may assume
$$\text{Ind}(G)=\langle F_{11},\ldots,F_{1b}\rangle\ast\langle F_{21},\ldots,F_{2b}\rangle\ast\cdots\ast\langle F_{a1},\ldots,F_{ab}\rangle,$$
where $|F_{ij}|=k$ and $F_{ij}\cap F_{kl}=\emptyset$ for all $1\leq i,k\leq a$ and $1\leq j,l\leq b$.\\
If $k>1$, then $F=F_{11}\cup F_{21}\cup\ldots\cup F_{(a-1)1}$ is a face of $\text{Ind}(G)$ whose link is $\rm{link}_\Delta(F)=\text{Ind}(G_a)=\langle F_{a1},\ldots,F_{ab}\rangle.$ This yields that $\rm{dim}\rm{link}_\Delta(F)=k-1>0$ and $\rm{link}_\Delta(F)$ is disconnected since $b>1$. Thus, if $G$ is $S_2$, then $k=1$, i.e., $n=ab$. Corollary \ref{COR}.(i) now completes the proof.
\end{proof}

As an immediate consequence, we get the following which is the $S_2$ analogue of \cite[Theorem 5.9]{EVV}.

\begin{cor} \label{COR1}
Let $G=C(n;1,b)=C(mb;1,b)$ be a one-paired circulant graph. Then $G$ is $S_2$ if and only if $m=1$, i.e., $G$ is Cohen-Macaulay.
\end{cor}

\begin{exam} {\rm It follows from Corollaries \ref{COR1} and \ref{COR}.(ii) that, if $m>1$, then $C(mb;1,b)$ is Buchsbaum but not $S_2$.}
\end{exam}

\section{circulant cubic graphs}

A graph in which every vertex has degree $3$, is called a cubic graph. It is easy to see that a circulant cubic graph is of the form $C_{2n}(a,n)$ with $1\leq a <n$. Brown and Hoshino \cite[Theorem 4.3]{BH2} characterized which connected circulant cubic graphs are well-covered. They showed that a connected circulant cubic graph $G$ is well-covered if and only if it is isomorphic to one of the following graphs: $C_4(1,2)$, $C_6(1,3)$, $C_6(2,3)$, $C_8(1,4)$, or $C_{10}(2,5)$ (see, \cite[Theorem 4.3]{BH2}).

On the other hand, Vander Meulen {\it et al.} in \cite[Theorem 5.2]{VVW} proved that a connected circulant cubic graph $G$ is Cohen-Macaulay if and only if it is isomorphic to $C_4(1,2)$ or $C_6(2,3)$. In the next Proposition we examine which connected circulant cubic graphs are $S_2$.

\begin{prop} \label{PROP1} The only non-$S_2$ connected circulant cubic graph is $C_6(1,3)$.
\end{prop}

\begin{proof} We know that $C_4(1,2)$ and $C_6(2,3)$ are Cohen-Macaulay. If $G=C_8(1,4)$, then $\text{Ind}(G)=\langle 025,035,036,136,147,247,257\rangle$. Thus $\text{Ind}(G)$ and $\text{link}_{\text{Ind}(G)}(0)$ are connected, i.e., $G$ is $S_2$. If $G=C_{10}(2,5)$, then
$$\text{Ind}(G)=\langle 0147,0347,0367,0369,1258,1458,1478,2369,2569,2589\rangle.$$
Again $\text{Ind}(G)$ and $\text{link}_{\text{Ind}(G)}(0)$ are connected. One can easily check that $\text{link}_{\text{Ind}(G)}(F)$ is connected for each $F\in\rm{Ind}(G)$ with $|F|=1$. Therefore $C_{10}(2,5)$ is also $S_2$. Finally, note that the independence complex of $C_6(1,3)$ is a disconnected two dimensional simplicial complex, and so it is not $S_2$.
\end{proof}

\begin{exam} {\rm Earl, Vander Meulen, and Van Tuyl showed that all connected circulant cubic graphs are Buchsbaum (see \cite[Table 1]{EVV}). By Proposition \ref{PROP1}, $C_6(1,3)$ is the only connected circulant cubic Buchsbaum graph which is not $S_2$.}
\end{exam}

We need some preliminaries to generalize Proposition \ref{PROP1} to all circulant graphs.

\begin{prop} \label{PROP2}
Let $\Delta_1$ and $\Delta_2$ be simplicial complexes with disjoint vertex sets. Then $\Delta=\Delta_1\ast\Delta_2$ is $S_2$ if and only if $\Delta_1$ and $\Delta_2$ are $S_2$.
\end{prop}

\begin{proof}  First note that if $\Delta_1$ and $\Delta_2$ are nonempty simplicial complexes, then $\Delta_1\ast\Delta_2$ is connected. On the other hand, it is not difficult to check that for $F_1\in\Delta_1$ and $F_2\in\Delta_2$ one has $\text{link}_\Delta(F_1\cup F_2)=\text{link}_{\Delta_1}(F_1)\ast\text{link}_{\Delta_2}(F_2)$.\\
Now assume $\Delta$ is $S_2$, and $F_1\in\Delta_1$ be such that $\text{dimlink}_{\Delta_1}(F_1)>0$. Then for a facet $F_2\in\Delta_2$ one has $F=F_1\cup F_2\in \Delta$ and $\text{link}_\Delta(F)=\text{link}_{\Delta_1}(F_1)\ast\text{link}_{\Delta_2}(F_2)=\text{link}_{\Delta_1}(F_1)\ast\{\emptyset\}=\text{link}_{\Delta_1}(F_1)$. Since $\Delta$ is $S_2$, $\text{link}_\Delta(F)$ is connected, and so is $\text{link}_{\Delta_1}(F_1)$, i.e., $\Delta_1$ is $S_2$. Similarly $\Delta_2$ is $S_2$.\\
Conversely, assume $\Delta_1$ and $\Delta_2$ are $S_2$. Also assume $F\in\Delta$ be such that $\text{dimlink}_{\Delta}(F)>0$. If $F\in\Delta_1$ (similarly for $F\in\Delta_2$), then $\text{link}_\Delta(F)=\text{link}_{\Delta_1}(F)\ast\text{link}_{\Delta_2}(\emptyset)=\text{link}_{\Delta_1}(F)\ast\Delta_2$. Now if $\text{link}_{\Delta_1}(F)\neq\emptyset$, then $\text{link}_{\Delta_1}(F)\ast\Delta_2$ is always connected, and if $\text{link}_{\Delta_1}(F)=\emptyset$, then $\text{link}_{\Delta}(F)=\Delta_2$ which is connected because it is $S_2$ (with positive dimension) and we are done. Now assume $F=F_1\cup F_2$ where  $F_1\in\Delta_1$ and $F_2\in\Delta_2$, and that $F_1\neq\emptyset$ and $F_2\neq\emptyset$. Thus $\text{link}_\Delta(F)=\text{link}_{\Delta_1}(F_1)\ast\text{link}_{\Delta_2}(F_2)$. If $\text{link}_{\Delta_1}(F_1)$ and $\text{link}_{\Delta_2}(F_2)$ are nonempty, then $\text{link}_{\Delta}(F)$ is connected. If $\text{link}_{\Delta_1}(F_1)=\emptyset$, then $\text{link}_\Delta(F)=\text{link}_{\Delta_2}(F_2)$ which is connected because $\Delta_2$ is $S_2$.
\end{proof}

\begin{cor} \label{COR2}
Let $G$ be a simple graph which is a disjoint union of two graphs $H$ and $K$. Then $G$ is $S_2$ if and only if $H$ and $K$ are $S_2$.
\end{cor}

\begin{proof}  It is enough to notice that $\text{Ind}(G)=\text{Ind}(H)\ast\text{Ind}(K)$, and apply Lemma \ref{PROP2}.
\end{proof}

We also need the following result of Davis and Domke \cite{DD} to extend Proposition \ref{PROP1} to all circulant cubic graphs.

\begin{thm} \label{THM5}
Let $G=C_{2n}(a,n)$ with $1\leq a<n$, and let $t={\rm gcd}(a,2n)$.
\begin{itemize}
\item[(i)] If $\frac{2n}{t}$ is even, then $G$ is isomorphic to $t$ copies of $C_{\frac{2n}{t}}(1,\frac{n}{t})$.
\item[(ii)] If $\frac{2n}{t}$ is odd, then $G$ is isomorphic to $\frac{t}{2}$ copies of $C_{\frac{4n}{t}}(2,\frac{2n}{t})$.
\end{itemize}
\end{thm}

Now we bring the main result of this section.

\begin{thm} Let $G=C_{2n}(a,n)$ be a circulant cubic graph and let $t={\rm gcd}(a,2n)$. Then $G$ is $S_2$ if and only if $\frac{2n}{t}=3,4,5,8$.
\end{thm}

\begin{proof} First assume $\frac{2n}{t}$ is even. By Theorem \ref{THM5} and Corollary \ref{COR2}, $G$ is $S_2$ if and only if $\frac{2n}{t}=4,8$.
Now assume $\frac{2n}{t}$ is odd. Again by Theorem \ref{THM5} and Corollary \ref{COR2}, $G$ is $S_2$ if and only if $\frac{4n}{t}=6,10$, or equivalently, $\frac{2n}{t}=3,5$.
\end{proof}

\begin{exam} {\rm $G=C_{16}(2,8)$ is an example of a circulant graph which is $S_2$ but not Buchsbaum. Indeed, $C_{16}(2,8)$ is isomorphic to $2$ copies of $C_{8}(1,4)$, and hence it is $S_2$, but as it is shown in \cite[Table 1]{VVW}, it is not Buchsbaum.}
\end{exam}

As the final result of this paper, we present infinite families of circulant graphs which are $S_2$ but not Buchsbaum. To do this, we need the following lemma.

\begin{lem} \label{LEM3}
{\rm (\cite[Lemma 2.5]{EVV})}  Let $G$ and $H$ be two disjoint graphs that are both Buchsbaum, but not
Cohen-Macaulay. Then $G\cup H$ is not Buchsbaum.
\end{lem}

\begin{cor} For $t>1$, the followings are infinite families of circulant graphs which are $S_2$ but not Buchsbaum:
\begin{itemize}
\item[(i)] $C_{8t}(t,4t)$.
\item[(ii)] $C_{10t}(2t,5t)$ and $C_{10t}(4t,5t)$.
\end{itemize}
\end{cor}

\begin{proof} (i) By Theorem \ref{THM5}.(i), $C_{8t}(t,4t)$ is isomorphic to $t>1$ copies of $C_{8}(1,4)$, and hence it is $S_2$. On the other hand, $C_{8}(1,4)$ is Buchsbaum but not Cohen-Macaulay. Thus Lemma \ref{LEM3} implies that $C_{8t}(t,4t)$ is not Buchsbaum.\\
(ii) By Theorem \ref{THM5}.(ii), $C_{10t}(2t,5t)$ and $C_{10t}(4t,5t)$ are isomorphic to $t>1$ copies of $C_{10}(2,5)$, and hence they are $S_2$. Again, $C_{10}(2,5)$ is Buchsbaum but not Cohen-Macaulay. Now Lemma \ref{LEM3} yields that $C_{8t}(t,4t)$ is not Buchsbaum.
\end{proof}

\begin{center}

%{\bf ACKNOWLEDGMENT}

\end{center}
\noindent

\providecommand{\bysame}{\leavevmode\hbox
to3em{\hrulefill}\thinspace}

\bigskip
\bigskip

\end{document}